\documentclass[12pt]{amsart}

\usepackage{times}
\usepackage{amsthm}
\usepackage{amssymb}
\usepackage{amsmath}
\usepackage{amsfonts} 
\usepackage{amsthm}
\usepackage{bm}
\usepackage{color}
\usepackage{graphicx}
\usepackage[all]{xy}

\newcommand{\excise}[1]{}

\newtheorem{theorem}{Theorem}[section]
\newtheorem{lemma}[theorem]{Lemma}

\newtheorem{corollary}[theorem]{Corollary}
\newtheorem{proposition}[theorem]{Proposition}

\theoremstyle{definition}
\newtheorem{example}[theorem]{Example}
\newtheorem{remark}[theorem]{Remark}

\newtheorem{definition}[theorem]{Definition}
\newtheorem{convention}[theorem]{Convention}

\newtheorem{notation}[theorem]{Notation}

        {\begin{list}
                {\noindent\makebox[0mm][r]{\arabic{enumi}.}}
                {\leftmargin=5.5ex \usecounter{enumi}}
        }
        {\end{list}}

        {\begin{list}
                {\noindent\makebox[0mm][r]{(\roman{enumi})}}
                {\leftmargin=5.5ex \usecounter{enumi}}
        }
        {\end{list}}

\def\<{\langle}
\def\>{\rangle}
\def\0{\mathbf{0}}
\def\CC{{\mathbb C}}

\def\NN{{\mathbb N}}

\def\QQ{{\mathbb Q}}

\def\VV{{\mathbb V}}

\def\ZZ{{\mathbb Z}}

\def\kk{\Bbbk}
\def\mm{{\mathfrak m}}

\def\kk{{\Bbbk}}

\def\del{\partial}

\def\qdeg{{\rm qdeg}}
\def\tdeg{{\rm tdeg}}
\def\Span{{\rm Span}}

\def\rank{{\rm rank\ }}
\def\cocoa{{\hbox{\rm C\kern-.13em o\kern-.07em C\kern-.13em o\kern-.15em A}}}

\def\image{{\rm image\ }}
\def\codim{{\rm codim}}
\def\link{{\rm lk\ }}

\def\minus{\smallsetminus}
\def\nothing{\varnothing}

\def\ol#1{{\overline {#1}}}
\def\wt#1{{\widetilde {#1}}}

\numberwithin{equation}{section}
\begin{document}

\mbox{}
\vspace{-3ex}
\title{A-graded methods for monomial ideals} 

\author{Christine Berkesch}
\address{Department of Mathematics \\ Purdue University \\
West Lafayette, IN 47907}
\email{cberkesc@math.purdue.edu}
\thanks{CB was partially supported by NSF Grant DMS 0555319}

\author{Laura Felicia Matusevich}
\address{Department of Mathematics \\
Texas A\&M University \\ College Station, TX 77843.}
\email{laura@math.tamu.edu}
\thanks{LFM was partially supported by NSF Grant DMS 0703866 and a Sloan
Research Fellowship}


\begin{abstract}
We use $\ZZ^d$-gradings to study $d$-dimensional monomial ideals.
The Koszul functor is employed to interpret the quasidegrees
of local cohomology in terms of the geometry of distractions
and to explicitly compute the multiplicities of exponents.
These multigraded techniques originate from the
study of hypergeometric systems of differential equations.
\end{abstract}
\maketitle

\mbox{}
\vspace{-5.05ex}
\parskip=0ex
\parindent2em
\setcounter{tocdepth}{2}
\parskip=1ex
\parindent0pt

\section{Introduction}
\label{sec:intro}

A cornerstone of combinatorial commutative algebra is the
connection between simplicial complexes and
ideals generated by squarefree monomials, also known as
\emph{Stanley--Reisner ideals}.
One of the fundamental results in this theory is
Reisner's criterion (Theorem~\ref{thm:reisner}),
which expresses the Cohen--Macaulayness of a
Stanley--Reisner ring as
a topological condition on the corresponding
simplicial complex.

These ideas can be adapted to study monomial ideals $I$
that are not squarefree. 
For instance, \emph{polarization} produces
a squarefree monomial ideal that shares many important properties
with~$I$. However, this construction introduces (many) new variables, 
making it impractical from a computational standpoint 
(see Example~\ref{ex:polarization}).

In contrast, the \emph{distraction} of a monomial ideal $I$ 
(used in Hartshorne's work on the Hilbert scheme \cite{Hartshorne paper})
has many of the features of its polarization, without the increase in dimension. 
While the distraction of $I$
is not a monomial ideal unless $I$ is squarefree, it can be studied
using a finite family of simplicial complexes,
called \emph{exponent complexes} (see Definition~\ref{def:deltab}). 

Distractions arise naturally in the algebraic study of differential equations.
Let $I$ be a monomial ideal in the ring
$\CC[\del] = \CC[\del_1,\dots,\del_n]$, a commutative polynomial subalgebra 
of the Weyl algebra $D_n$ of linear partial differential operators on 
$\CC[x_1,\dots, x_n]$, where $\del_i$ denotes the operator ${\del}/{\del x_i}$.
Let $\theta_i = x_i \del_i$, 
so that $\CC[\theta] = \CC[\theta_1,\dots,\theta_n]$ is also 
a commutative polynomial subalgebra of $D_n$.
The \emph{distraction} of $I \subseteq \CC[\del]$ is the ideal
\[
\tilde{I} =  
(\CC(x) \otimes_{\CC[x]} D \cdot I) \cap \CC[\theta] \subseteq \CC[\theta].
\] 

Let $d = \dim(\CC[\del]/I)$. Any integer matrix $A$
whose columns $a_1,\dots,a_n$ span $\ZZ^d$ as a lattice
induces a $\ZZ^d$-grading on $D_n$ 
via $\tdeg(\del_j) = a_j = - \tdeg(x_j)$.  
Monomial ideals 
and their distractions 
are $A$-homogeneous and 
have the same holomorphic solutions 
when considered as systems of differential equations. 
While this solution space is 
usually infinite dimensional, 
its subspace of $A$-homogeneous solutions 
of any particular degree is finite dimensional (see Definition~\ref{def:rank}). 
One captures these solutions
by adding 
\emph{Euler operators} to these ideals, 
where a chosen degree is viewed as a parameter.

In this article we compute the dimension of the 
$A$-homogeneous holomorphic solutions 
of $I$ as a function of the parameters. 
Our starting point is Theorem~\ref{thm:EK-homology}, which 
provides a strong link between the Koszul homology of 
$\CC[\theta]/\tilde{I}$ with respect to the Euler operators  
and the $A$-graded structure of the local cohomology of $\CC[\del]/I$
at the maximal ideal. 
The background necessary for this result 
occupies Section~\ref{sec:apply-mmw}. 
In Section~\ref{sec:S-R expts}, we use the exponent complexes of $\tilde{I}$  
to provide a new topological criterion for the 
Cohen--Macaulayness of the monomial ideal $I$, Theorem~\ref{thm:main}. 
Finally, in Section~\ref{sec:primary complex}, 
we give a combinatorial formula for the dimension of the 
$A$-homogeneous holomorphic solutions of $I$ of a fixed degree, 
which may be viewed as an intersection multiplicity
(see Theorem~\ref{thm:dim count squarefree}~and
Proposition~\ref{thm:dim count general}).

This work was inspired by the theory
of $A$-hypergeometric differential equations, whose
intuition and techniques are employed throughout, 
especially those developed
in \cite{SST,MMW,DMM,rank jumps}. 
As we apply results from partial differential equations 
in Sections~\ref{sec:apply-mmw} and~\ref{sec:S-R expts},
we use the ground field $\CC$. 
Our techniques in Section~\ref{sec:primary complex} 
are entirely homological and thus work over any field of characteristic zero.
Since our findings were first circulated,
Ezra Miller \cite{miller-letter} has found alternative 
commutative algebra proofs for the results in Section~\ref{sec:S-R expts},
as well as a version of Theorem~\ref{thm:EK-homology}, which are valid over
an arbitrary field.

\subsection*{Acknowledgments}
We are grateful to Uli Walther for his hospitality while the
second author visited Purdue University, and for thoughtful comments
that improved a previous version of this article. We also thank
Giulio Caviglia, who pointed us to the polarization of monomial ideals,
Vic Reiner, for crucial advice on simplicial matroids that benefited 
Section~\ref{sec:primary complex}, 
and Bernd Sturmfels, for his encouraging comments.
We especially thank Ezra Miller,
who made us aware of important references,
and offered many valuable suggestions, including
the clarification of Remark~\ref{remark:finitely many}.

\section{Differential operators, local cohomology, and $\ZZ^d$-gradings}
\label{sec:apply-mmw}

Although our results can be stated in commutative terms, our
methods and ideas come from the study of hypergeometric differential
equations. Thus, we start by introducing the \emph{Weyl algebra} 
$D = D_n$ of linear partial differential operators with polynomial coefficients:
\[
D = \CC\left\< x, \del \; \bigg\vert \;  
              [\del_i,x_j]=\delta_{ij}, [x_i,x_j]=0=[\del_i,\del_j] 
       \right\>,
\]
where $x = x_1,\dots,x_n$ and $\del=\del_1,\dots,\del_n$.
We distinguish two commutative polynomial subrings of $D$, namely
$\CC[\del] = \CC[\del_1,\dots,\del_n]$ and
$\CC[\theta]=\CC[\theta_1,\dots,\theta_n]$, where
the $\theta_i = x_i\del_i$, 
and we denote $\mm = \< \del_1,\dots, \del_n\> \subseteq \CC[\del]$.

\begin{convention}
\label{conv:I}
Henceforth, $I$ is a monomial ideal in $\CC[\del]$,
and $d$ denotes the Krull dimension of the 
ring $\CC[\del]/I$.
\end{convention}

\begin{definition}
\label{def:distraction}
Given any left $D$-ideal $J$, the \emph{distraction} of $J$ is the ideal
\[
\tilde{J} = (\CC(x) \otimes_{\CC[x]} D \cdot J )\cap \CC[\theta]
\subseteq \CC[\theta].
\]
For the monomial ideal $I \subseteq \CC[\del]$,
the identity
$x_i^m \del_i^m = \theta_i(\theta_i - 1) \cdots (\theta_i - m + 1)$ implies that
\begin{equation}
\label{eqn:distraction}
\tilde{I} = \wt{(D\cdot I)}  =
\left\<
    [\theta]_u :=\prod_{i=1}^n \prod_{j=0}^{u_i-1} (\theta_i-j) \ 
                             \bigg\vert \ \del^u \in I
\right\>.
\end{equation}
\end{definition}

\begin{remark}
We emphasize that $I$ and $\tilde{I}$ live in different 
$n$-dimensional polynomial rings. 
Also, note that the zero set of the distraction 
$\tilde{I}$ is the Zariski closure
in $\CC^n$ of the exponent vectors of the standard monomials of $I$.
\end{remark}

\begin{remark}
We can think of $I$ and $\tilde{I}$ either as systems of partial differential
equations or as systems of polynomial equations, and this dichotomy will
be useful later on. To avoid confusion, a \emph{solution}
of $I$ or $\tilde{I}$ will always be a holomorphic function,
while the name \emph{zero set} (and the notation $\VV(I)$, $\VV(\tilde{I})$) 
is reserved for an algebraic variety in $\CC^n$.
\end{remark}

A monomial ideal in $\CC[\del]$ is automatically homogeneous with
respect to every natural grading of the polynomial ring,
from the usual (coarse) $\ZZ$-grading
to the finest grading by $\NN^n$.
In this article we take the middle road and use a $\ZZ^d$-grading, 
where $d = \dim(\CC[\del]/I)$.

\begin{definition}
\label{def:qdeg}
A $\ZZ^d$-grading, called an $A$-grading, 
of the Weyl algebra $D$
is determined by a matrix
$A \in \ZZ^{d\times n}$, or more precisely by its columns
$a_1,\dots,a_n \in \ZZ^d$ via
\[
\tdeg(x_i) = -a_i, \quad \tdeg(\del_i)=a_i
\]
Given an $A$-graded $D$-module $M$, the set of \emph{true degrees}
of $M$ is
\[
\tdeg(M) =
\{ \alpha \in \ZZ^d \mid M_{\alpha} \neq 0\} \subseteq \ZZ^d.
\]
The set of \emph{quasidegrees} of $M$, denoted by $\qdeg(M)$, is the
Zariski closure of $\tdeg(M)$ under the natural inclusion
$\ZZ^d \subseteq \CC^d$.
\end{definition}

The notions of $\tdeg$ and $\qdeg$ arise from the $A$-grading
of $D$, and should not be confused with the
degree $\deg$ of a homogeneous ideal in a polynomial ring
that is computed using the Hilbert polynomial, as in the following definition.

\begin{definition}
The Hilbert polynomial of $\CC[\del]/I$ 
(with respect to the standard $\ZZ$-grading on
$\CC[\del]$) has the form $P_I(z) = \frac{m}{d!} z^d + \cdots$.
The \emph{degree} of $I$, denoted by $\deg(I)$
is the number $m$,
which is equal to the number of
intersection points (counted with multiplicity) of $\VV(I)$
and a sufficiently generic affine space of dimension $n - d$.
\end{definition}

The local cohomology modules of $\CC[\del]/I$ with respect to 
the ideal $\mm=\< \del_1,\dots, \del_n\>$ are also $A$-graded.
We describe the quasidegrees of 
$\bigoplus_{i<d} H^i_{\mm}(\CC[\del]/I)$
in Theorem~\ref{thm:EK-homology}.
First, we make precise the conditions
required of our grading matrix $A$.

\begin{convention}
\label{conv:A}
Let $I \subseteq \CC[\del]$ be a monomial ideal of Krull dimension $d$.
For the remainder of this article,
we fix a $d\times n$ integer matrix $A$ such that
\begin{enumerate}
\item $\ZZ A = \ZZ^d$,
\item $\NN A \cap (-\NN A) = 0$,
\item For each $\sigma \subseteq \{ 1,\dots ,n\}$,
$\dim_{\QQ}(\Span\{a_{i} \mid i \in \sigma\}) = \min(d,|\sigma|)$, 
where $|\sigma|$ is the cardinality of $\sigma$.
\end{enumerate}
\end{convention}

\begin{notation}
\label{not:E}
We denote by $E_i$ the linear form
$\sum_{j=1}^n a_{ij} \theta_j \in \CC[\theta]$,
where $A=(a_{ij})$ is our matrix from Convention~\ref{conv:A}.
For $\beta \in \CC^d$, let $E-\beta$ denote the sequence
$\{ E_i - \beta_i \}_{i=1}^d$ of \emph{Euler operators}; 
in particular, $E$ is the sequence $E_1,\dots,E_d$.
\end{notation}

In Convention~\ref{conv:A}, the first condition ensures that our
grading group is all of $\ZZ^d$, the second that
the cone spanned by the columns of $A$ is \emph{pointed},
that is, contains no lines. The third condition
is used in the computations in Section~\ref{sec:primary complex}, and 
implies that 
$\CC[\theta]/(\tilde{I} + \< E -\beta \>)$ is Artinian for all $\beta \in \CC^d$. 
Further, it implies that the sequence of polynomials
given by the coordinates of $A\cdot \del$, for
$\del$ the column vector $[\del_1,\dots,\del_n]^t$,
is a linear system of parameters for $I$. 
A generic $d\times n$ integral matrix will satisfy these requirements.

\medskip

\emph{Given a $d$-dimensional
monomial ideal $I \subseteq \CC[\del]$ 
as in Convention~\ref{conv:I},
a matrix $A$ as in Convention~\ref{conv:A}, 
and a parameter vector $\beta \in \CC^d$,
our main object of study is the left $D$-ideal
\[
I + \< E - \beta \> \subseteq D. 
\]
}

\smallskip

As with all left $D$-ideals, $I+\< E-\beta\>$ is a system of 
partial differential equations, so we may consider its space of 
germs of holomorphic solutions at a generic nonsingular point in $\CC^n$.

\begin{definition}
\label{def:rank}
By the choice of $A$ in Convention~\ref{conv:A}, $I+\< E-\beta\>$
is \emph{holonomic} (see \cite[Definition~1.4.8]{SST}), so the
dimension of this solution space, 
called the \emph{holonomic rank} of the system 
and denoted $\rank(I+\<E-\beta\>)$, is finite. 

If $\phi(x)$ is a solution of $I +\< E-\beta\>$ and
$t\in (\CC^*)^d$,
then
$\phi(t^{a_1}x_1,\dots,t^{a_n}x_n) = t^{\beta}\phi(x)$.
This justifies the claim in the Introduction:
$\rank(I+\<E-\beta\>)$ is the dimension of the
space of holomorphic solutions of 
$I$ that are $A$-homogeneous of degree $\beta$.
\end{definition}

\begin{lemma}
\label{lemma:rank}
Let $I \subseteq \CC[\del]$ be a monomial ideal and
$\tilde{I} \subseteq \CC[\theta]$ its distraction.
Then there is an equality of ranks: 
$\rank(I+\< E-\beta\>) = \rank(\tilde{I} + \< E-\beta\>)$.
\end{lemma}

\begin{proof}
Recall that $\tilde{I} = \< [\theta]_u \mid \del^u \in I \>$
and $[\theta]_u = x^u\del^u$. The result then follows, as multiplying
operators on the left by elements of $\CC[x]$
does not change their holomorphic solutions.
\end{proof}

Recall that the generators of $\tilde{I}+\< E-\beta\>$ lie in $\CC[\theta]$. 
Left $D$-ideals with this property are called \emph{Frobenius ideals}, 
and their holonomicity can by checked through commutative algebra.

\begin{proposition}\cite[Proposition~2.3.6]{SST}
\label{prop:frobenius}
Let $J$ be an ideal in $\CC[\theta]$. The left
$D$-ideal $D\cdot J$ is holonomic if and only if
$\CC[\theta]/J$ is an Artinian ring. In this case,
$\rank(D\cdot J) = \dim_{\CC}(\CC[\theta]/J)$.
\end{proposition}

By our choice of $A$, $\tilde{I}+\<E-\beta\>$
satisfies the hypotheses of Proposition~\ref{prop:frobenius}.
Combining this with Lemma~\ref{lemma:rank} 
and the fact that holonomicity of $\tilde{I}+\<E-\beta\>$ 
implies holonomicity of $I+\< E-\beta\>$,
we obtain the following.

\begin{lemma}
\label{lemma:all-holonomic}
Let $I \subseteq \CC[\del]$ be a monomial ideal as in Convention~\ref{conv:I}
and $A$ an integer matrix as in Convention~\ref{conv:A}. 
The left $D$-ideal $I+\<E-\beta\>$ is holonomic for all $\beta$ and
\begin{equation}
\label{eqn:rank=dim}
\rank(I+\< E-\beta\>) =
\dim_{\CC} \left( \frac{\CC[\theta]}{\tilde{I}+\< E-\beta \>}
  \right).
\end{equation}
\end{lemma}

The right hand side of~(\ref{eqn:rank=dim})
appears again in Corollaries~\ref{thm:local-coh-degrees} 
and~\ref{cor:rank CM} below.
We will compute this rank explicitly
in Section~\ref{sec:primary complex}; 
we observe now 
that $\dim_{\CC}(\CC[\theta]/(\tilde{I} + \< E-\beta \>)) = \deg(I)$ if and
only if the Koszul complex $K_\bullet(\CC[\theta]/\tilde{I};E-\beta)$
has no higher homology. We use the notation
$H_\bullet(\CC[\theta]/\tilde{I};\beta)$ for the homology 
of $K_\bullet(\CC[\theta]/\tilde{I};E-\beta)$.

\begin{theorem}
\label{thm:EK-homology}
Let $I \subseteq \CC[\del]$ be a monomial ideal of dimension $d$,
$\mm = \< \del \>$, and $\tilde{I}\subseteq \CC[\theta]$
be the distraction of $I$.
The Koszul homology $H_i( \CC[\theta]/\tilde{I} ; \beta )$ is
nonzero for some $i>0$
if and only if
$\beta \in \CC^d$ is a quasidegree of
$\bigoplus_{j=0}^{d-1} H_\mm^j( \CC[\del]/I )$.
More precisely, if $k$ equals the smallest homological degree $i$ for which
$\beta \in \qdeg( H_\mm^i( \CC[\del]/I ) )$,
then $H_{d-k}( \CC[\theta]/\tilde{I} ; \beta )$ is holonomic of nonzero rank
while $H_i(  \CC[\theta]/\tilde{I} ; \beta ) = 0$ for $i > d - k$.
\end{theorem}
\begin{proof}
This is a combination of Theorems~4.6 and~4.9 \cite{DMM}, 
which are generalized versions of Theorems~6.6 and~8.2 in \cite{MMW}. 
Note that in order to apply these results, we need
the holonomicity of $I+\<E-\beta\>$ for all $\beta$,
which follows from Convention~\ref{conv:A} and
Lemma~\ref{lemma:all-holonomic}.
\end{proof}

\begin{corollary}
\label{thm:local-coh-degrees}
If $I \subseteq \CC[\del]$ is a monomial ideal, $d=\dim(\CC[t]/I)$, 
and $A$ is a $d\times n$ matrix as in Convention~\ref{conv:A}, then
\[
\left\{
\beta \in \CC^d \; \bigg\vert \;
     \dim_{\CC}\left(
          \frac{\CC[\theta]}{\tilde{I} + \< E -\beta \>}
          \right)
>
\deg(I)
\right\}
=
\qdeg
\left(
  \bigoplus_{i<d} H^i_{\mm}\left(\frac{\CC[\del]}{I} \right)
\right).
\]
\end{corollary}

\begin{proof}
This follows from
Theorem~\ref{thm:EK-homology}, as
the vanishing for all $i>0$
of the Koszul homology
$H_i( \CC[\theta]/\tilde{I} ; \beta )$ 
is equivalent to $\dim_\CC(\CC[\theta]/(\tilde{I}+\<E-\beta\>)) = \deg(I)$. 
\end{proof}

Since the vanishing of local cohomology characterizes Cohen--Macaulayness,
we have the following immediate consequence.

\begin{corollary}
\label{cor:rank CM}
Let $I \subseteq \CC[\del]$ be a monomial ideal and 
$A$ be as in Convention~\ref{conv:A}. 
The ring $\CC[\del]/I$ is Cohen--Macaulay if and only if
\begin{equation}
\label{eqn:CM-cond}
\dim_{\CC}\left( \frac{\CC[\theta]}{\tilde{I} + \< E - \beta\> }\right)
= \deg(I) \quad \quad \forall \beta \in \CC^d.
\end{equation}
\end{corollary}

\section{Simplicial complexes and Cohen--Macaulayness of monomial ideals}
\label{sec:S-R expts}

In Theorem~\ref{thm:main}, we give a combinatorial criterion to determine when the
ring $\CC[\del]/I$ is Cohen--Macaulay. Its statement 
is in the same spirit as Reisner's well-known result in the
squarefree case: one verifies Cohen--Macaulayness by checking
that certain simplicial complexes have vanishing homology.

Given $I \subseteq \CC[\del]$ a monomial ideal,
the zero set $\VV(\tilde{I}) \subseteq \CC^n$ is a union of
translates of coordinate spaces of the form
$\CC^{\sigma} = \{ v \in \CC^n \mid v_i = 0 \; \forall i \not\in \sigma\}$ 
for certain subsets $\sigma \subseteq \{1,\dots,n\}$. 
The irreducible decomposition of $\VV(\tilde{I})$ is controlled by 
combinatorial objects called \emph{standard pairs} of $I$, introduced in \cite{stv}.

\begin{definition}
\label{def:deltab}
For any $b \in \VV(\tilde{I})$, let
\[
\Delta_b(I) = \{ \sigma \subseteq \{1,\dots,n\}
\mid b + \CC^{\sigma} \subseteq \VV(\tilde{I}) \}.
\]
This is a simplicial complex on $\{1,\dots,n\}$, 
called an \emph{exponent complex}, 
whose facets correspond to the irreducible components 
of $\VV(\tilde{I})$ that contain $b$.  
For $\beta \in \CC^d$,
the \emph{exponents} of $I$ with respect to $\beta\in\CC^d$
are the elements of $\VV(\tilde{I} + \< E - \beta \>)$.
\end{definition}

Given an exponent $b$ of $I$, $x^b$ is a solution of the
differential equations $\tilde{I} + \< E - \beta \>$.
This justifies the name exponent complex for 
$\Delta_b(I)$, as any $b \in \VV(\tilde{I})$ is an
exponent corresponding to the parameter $\beta = A \cdot b$.
Exponent complexes were introduced in \cite[Section 3.6]{SST}
for the special case when the monomial ideal $I$ is an
initial ideal of the toric ideal $I_A$.

\begin{remark}
\label{remark:finitely many}
Only finitely many simplicial complexes can occur as 
exponent complexes of our monomial ideal $I$
because the number of vertices is fixed. 
To find the collection of all exponent complexes of $I$, 
it suffices to compute $\Delta_b(I)$ at the lattice points $b$  
in the closed cube with diagonal from $(-1,\dots,-1)$ to 
the exponent of the least common multiple of 
the minimal generators of $I$, 
commonly called the \emph{join} of their exponents. 
\end{remark}

\begin{figure}[h]
\includegraphics[angle=-90]{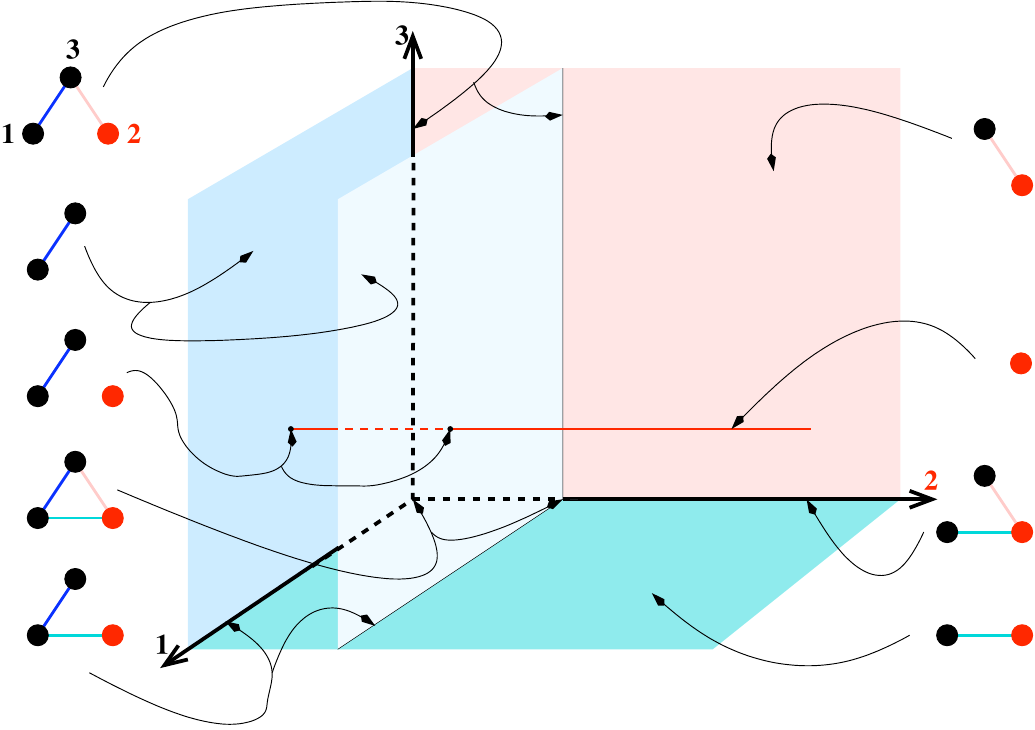}
\caption{The exponent complexes for
$I=\< \del_1^2\del_2^2\del_3,\del_1\del_2^2\del_3^2\>$.}
\label{fig1}
\end{figure}

\begin{example}
\label{ex:main}
The distraction $\tilde{I}$ of the
monomial ideal $I=\<\del_1^2\del_2^2\del_3,\del_1\del_2^2\del_3^2\>$
has five irreducible components:
\[
\CC^{\{1,2\}},\;\;
\CC^{\{2,3\}},\;\;
\CC^{\{1,3\}},\;\;
(0,1,0)+\CC^{\{1,3\}},\;\;
(1,0,1)+\CC^{\{2\}}.
\]

On the right and left of Figure~\ref{fig1} we have listed
the possible exponent complexes. The first complex on the 
left is associated to the
two vertical lines. The second is associated to the two 
$\{ 1,3\}$ planes, the third
is associated to the two points $(1,0,1)$ and $(1,1,1)$, 
the fourth is associated
to the two points $(0,0,0)$ and $(0,1,0)$, the fifth is associated to the two
lines parallel to the $1$st axis. On the right side, the first complex is
associated to the $\{2,3\}$ coordinate plane, the second is 
associated to all points
on the indicated line parallel to the $2$nd axis except for 
the previously mentioned ones,
the third complex is associated to the $2$nd axis, and the last complex is
associated to the $\{1,2\}$ coordinate plane.
\end{example}

\begin{example}
\label{ex:small}
Consider the monomial ideal
$I = \< \del_1^3 \del_2, \del_1^2 \del_2^2 \> \subseteq \CC[ \del_1, \del_2 ]$,
which has primary decomposition
$I = \< \del_1^2 \> \cap \< \del_2 \> \cap \< \del_1^3, \del_2^2 \>$.
The distraction of $I$ is 
\[
\tilde{I} = \< \theta_1(\theta_1-1)(\theta_1-2) \theta_2, 
\theta_1(\theta_1-1) \theta_2(\theta_2-1) \> 
\subseteq \CC[ \theta_1, \theta_2],
\] 
whose zero set $\VV(\tilde{I})$ is shown in Figure~\ref{fig:abc}(a), 
with the irreducible components labeled.

\begin{figure}[h]
\includegraphics[height=4.65cm]{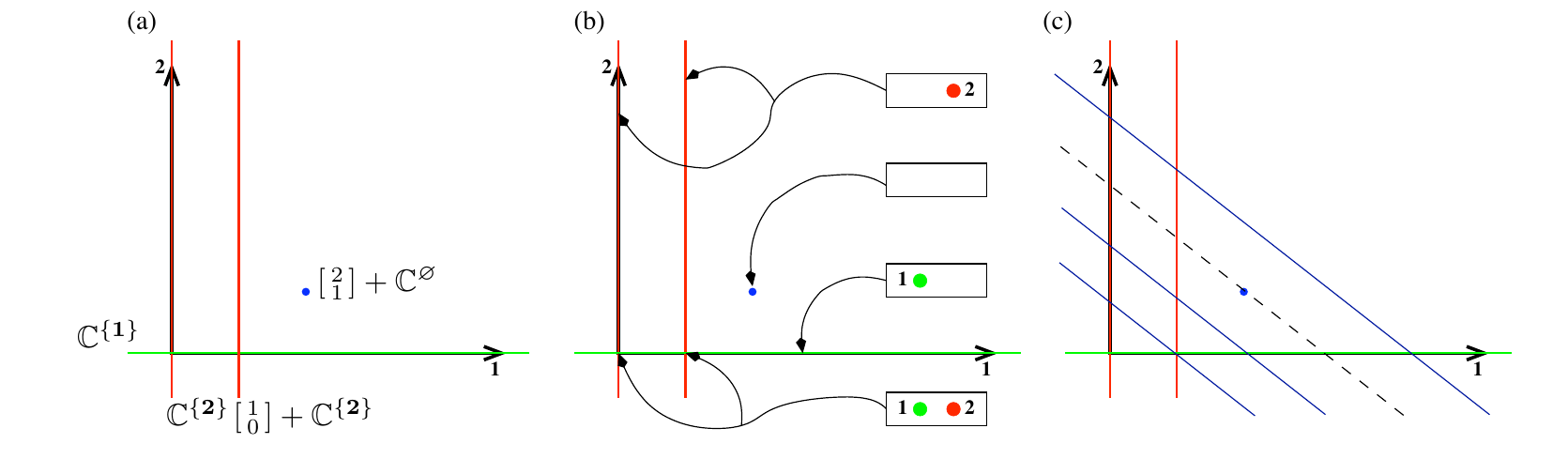}
\caption{$\VV(\tilde{I})$ for $I=\< \del_1^3 \del_2, \del_1^2 \del_2^2 \>$.}
\label{fig:abc}
\end{figure}

In this example, $\CC^{\{1\}}$ and
$
\left[ \begin{smallmatrix} 1 \\ 0 \end{smallmatrix} \right] + \CC^{ \{ 2 \} }
$
are the two components of
$\VV(\tilde{I})$ that contain the vector
$
b = \left[ \begin{smallmatrix} 1 \\ 0 \end{smallmatrix} \right]
$, 
see Figure~\ref{fig:abc}(b).
To this $b$ we associate the simplicial complex
$\Delta_b(I) = \{ \{ 1 \}, \{ 2 \}, \nothing \}$,
where the numbers 1 and 2 appear because a copy of $\CC$
appears in the corresponding components listed above.
Notice that all of the exponent complexes in Figure~\ref{fig:abc}(b)
are 0-dimensional, except for the one associated to 
$
\left[ \begin{smallmatrix} 2 \\ 1 \end{smallmatrix} \right]
$,
which is the empty set and arises because $I$ has an embedded prime.

The connection between the Cohen--Macaulayness of $I$ 
and the Euler operators 
can be expressed geometrically. 
The intersection of the variety $\VV(\tilde{I})$ with a
sufficiently general line consists of 3 points
(when counted with multiplicity). 
The solid slanted lines in Figure~\ref{fig:abc}(c)
represent a family of such lines, 
given by $\VV(\< E - \beta \>)$ when $A = [ \ 1 \ 1 \ ]$.
When $\beta = 3$, 
$\VV(\< E - \beta \>)$ (the dotted line in 
Figure~\ref{fig:abc}(c)) passes through the point 
$\left[ \begin{smallmatrix} 2\\1 \end{smallmatrix} \right]$. 
This results in an intersection multiplicity of 4 with $\VV(\tilde{I})$,
confirming the non-Cohen--Macaulayness of $I$.
\end{example}

For a general monomial ideal $I$, such a jump in intersection multiplicity
between $\VV(\tilde{I})$ and the 
$(n-d)$-dimensional linear space $\VV(\< E - \beta \>)$ 
can occur in two different ways. 
As we have just witnessed, if $I$ is not unmixed, a non-generic intersection
involving a lower dimensional component of
$\VV(\tilde{I})$ will have more points than expected,
as we have just seen.
Also, even if $I$ is unmixed, 
some non-generic points of
$\VV(\tilde{I})$ could contribute a higher multiplicity to the intersection
(see Example~\ref{ex:radCM}).
We show in Theorem~\ref{thm:Deltab CM} that this is the case
precisely at the points $b$ where $\Delta_b(I)$
fails to be a Cohen--Macaulay simplicial complex.

\begin{definition}
Given a simplicial complex $\Delta$ on $\{1,\dots,n\}$,
the \emph{Stanley-Reisner ring} or \emph{face ring} of $\Delta$
is $\CC[\theta]/I_{\Delta}$, where
\[
I_{\Delta} := \bigcap\ \< \theta_j : j \notin \sigma \>
\]
and the intersection runs over the facets $\sigma$ of $\Delta$.
Note that $I_{\Delta}$ is a squarefree monomial ideal, and
any squarefree monomial ideal can be obtained in this manner.
The Krull dimension of $\CC[\theta]/I_{\Delta}$
is the maximal cardinality of a facet of $\Delta$,
hence $\dim(\CC[\theta]/I_{\Delta}) = \dim(\Delta)+1$.
\end{definition}

\begin{definition}
\label{def:Jb}
For an exponent $b$ of $I$, let $J_b := I_{\Delta_b(I)} \subseteq \CC[\theta]$
denote the Stanley--Reisner ideal of the exponent complex $\Delta_b(I)$, and
$J_b(\theta - b)$ denote the ideal obtained from $J_b$ by replacing
$\theta_i$ with $\theta_i - b_i$.
Note that
\[
J_b(\theta - b)
= \bigcap \< \theta_i - b_i \mid i \notin \sigma \>,
\]
where the intersection runs over the facets of $\Delta_b(I)$, equivalently,
over the $\sigma \subseteq \{ 1,\dots,n \}$ such that $b+\CC^{\sigma}$ 
is an irreducible component of $\VV(\tilde{I})$.
\end{definition}

We now recall Reisner's criterion for the Cohen--Macaulayness of
Stanley--Reisner rings; proofs of this fact can be found in 
\cite[Corollary~5.3.9]{bruns-herzog}, \cite[Section II.4]{stanley-green-book}, 
and \cite[Section~13.5.2]{MS05}.

\begin{definition}
Let $\Delta$ be a simplicial complex and $\sigma\in\Delta$.
The \emph{link} of $\sigma$ is the simplicial complex
\[
\link \sigma =
\{ \tau \in \Delta \mid \sigma \cup \tau \in \Delta, \sigma \cap \tau = \nothing \}.
\]
\end{definition}
\vspace{.005cm}

\begin{theorem}
\label{thm:reisner}
Given a simplicial complex $\Delta$, its Stanley--Reisner
ring $\CC[\theta]/I_{\Delta}$
is Cohen--Macaulay
if and only if
$\wt{H}_i(\link \sigma;\CC) = 0$
for all $\sigma\in\Delta$ and for all $i < \dim ( \link \sigma)$.
\end{theorem}

\begin{definition}
\label{def:CM-complex}
A simplicial complex $\Delta$ is \emph{Cohen--Macaulay} if and only
if its Stanley--Reisner ring is Cohen--Macaulay; equivalently, when
$\Delta$ satisfies the condition in Theorem~\ref{thm:reisner}.
\end{definition}

We will use the following fact in the proof of Theorem~\ref{thm:Deltab CM}.

\begin{proposition}{\cite[Corollary 4.6.11]{bruns-herzog}}
\label{thm:lsop CM}
Let $J$ be a homogeneous ideal in $\CC[\theta]$.
If $E$ is a linear system of parameters for $\CC[\theta]/J$, then
\[
\deg(J) \leq \deg(J+\< E \>) = \dim_{\CC}(\CC[\theta]/(J+\< E \>)),
\]
and equality holds if and only if $J$ is Cohen--Macaulay.
\end{proposition}

The following result
is an adaptation of \cite[Theorem~4.6.1]{SST}.

\begin{theorem}
\label{thm:Deltab CM}
Given a monomial ideal $I\subseteq \CC[\del]$ and $\beta \in \CC^d$,
\[
\dim_\CC \left( \frac{\CC[\theta]}{\tilde{I} + \< E - \beta \>} \right) = \deg(I)
\]
if and only if $\Delta_b(I)$ is a Cohen--Macaulay complex of
dimension $d-1$ for all $b \in \VV(\tilde{I}+\<E - \beta\>)$.
\end{theorem}

\begin{proof}
Since 
$\tilde{I} +\< E-\beta\> = 
\bigcap_{b\in \VV(\tilde{I} + \< E - \beta \>)} J_b(\theta - b) + \< E-\beta\>$,
\[
\dim_{\CC}(\CC[\theta]/(\tilde{I}+\<E - \beta\>)) =
\sum_{b \in \VV(\tilde{I}+\<E - \beta \>)} \dim_{\CC}(\CC[\theta]/(J_b+\<E\>)).
\]
Proposition~\ref{thm:lsop CM} implies that 
for any $b\in\VV(\tilde{I} + \< E - \beta\>)$
such that $\dim(\CC[\theta]/(J_b)) = d$,
\[
\dim_{\CC}(\CC[\theta]/(J_b+\<E\>)) \geq \deg(J_b),
\]
with equality if and only if $J_b$ is Cohen--Macaulay.
By summing over the exponents $b$ of $I$ corresponding to $\beta$,
we obtain the inequality
\[
\dim_{\CC}(\CC[\theta]/(\tilde{I}+\<E - \beta\>))
=
\sum_{b \in \VV(\tilde{I}+\<E - \beta \>)} \dim_{\CC}(\CC[\theta]/(J_b+\<E\>))
\geq
\sum_{\dim(\CC[\theta]/J_b) = d} \deg(J_b).
\]
Notice that the right hand sum is
\[
\sum_{\dim(\CC[\theta]/J_b) = d} \deg(J_b) = \deg(I)
\]
because the degree of a monomial ideal is equal to
the number of top dimensional irreducible components
of its distraction.
Therefore
\[
\dim_{\CC} \left( \frac{\CC[\theta]}{\tilde{I}+\<E - \beta\>} \right) 
\geq \deg(I),
\]
with equality if and only if
for all $b\in\VV(\tilde{I}+\<E - \beta\>)$ with 
$\dim(\CC[\theta]/J_b) = d$, we have that 
$\dim_{\CC}(\CC[\theta]/(J_b+\<E\>)) = \deg(J_b)$.
If $\dim(\CC[\theta]/J_b)=d$, then $E$ is a linear system of parameters for 
$J_b$ and
\[
\dim_{\CC}(\CC[\theta]/(J_b+\<E\>)) \geq \deg(J_b),
\]
with equality if and only if $\CC[\theta]/J_b$ is
Cohen--Macaulay, by Proposition~\ref{thm:lsop CM}.
Hence we have
$\dim_{\CC}(\CC[\theta]/(\tilde{I}+\<E-\beta\>)) = \deg(I)$
exactly when
$\CC[\theta]/J_b$ is Cohen--Macaulay of Krull dimension $d$
for all $b \in \VV(\tilde{I}+\< E - \beta \>$.
Finally, recall that by definition
$\Delta_b(I)$ is Cohen--Macaulay of dimension $d - 1$
when $J_b = I_{\Delta_b(I)}$ is Cohen--Macaulay of dimension $d$.
\end{proof}

The main result of this section
now follows directly from Corollary~\ref{cor:rank CM} and
Theorem~\ref{thm:Deltab CM}. 
Recall from Remark~\ref{remark:finitely many} that there are finitely many 
exponent complexes $\Delta_b(I)$ that can be obtained after 
consideration of finitely many exponents $b\in \ZZ^n$.

\smallskip

\begin{theorem}
\label{thm:main}
A $d$-dimensional monomial ideal $I$ is Cohen--Macaulay if and only if
its (finitely many) 
exponent complexes are Cohen--Macaulay of dimension $d-1$. 
\end{theorem}

\smallskip

\begin{remark}
If $I=I_{\Delta}$ is squarefree and $\Delta_0(I)=\Delta$ is Cohen--Macaulay, 
then Theorem~\ref{thm:reisner} and Theorem~\ref{thm:main} 
together imply that all exponent complexes of $I$ are Cohen--Macaulay of
dimension $d-1$.
\end{remark}

\begin{example}[Example~\ref{ex:main}, continued]
The third simplicial complex on the left in 
Figure~\ref{fig1} is not Cohen--Macaulay
by Theorem~\ref{thm:reisner}
because the link of the empty set is one-dimensional and
$\wt{H}_0(\link \nothing ; \CC) = \CC \neq 0$.
It now follows from Theorem~ \ref{thm:main} that $I$ is not Cohen--Macaulay.
\end{example}

\begin{example}[Example~\ref{ex:small}, continued]
\label{ex:small-continued}
For $b = \left[ \begin{smallmatrix} 2\\1 \end{smallmatrix} \right]$, 
$\Delta_b(I)$ is empty, while $d = 2$. 
Thus by Theorem~\ref{thm:main}, $I$ is not Cohen--Macaulay. 
\end{example}

It is well-known that the Cohen--Macaulay property of 
a monomial ideal is inherited
by its radical (see \cite[Proposition~3.1]{taylor} or \cite[Theorem~2.6]{HTT});
however, the converse is not true. 
Theorem~\ref{thm:main} provides the conditions necessary to obtain a converse,
as $\Delta_0(I)$ is the Stanley--Reisner complex for the radical of $I$. 

\begin{example}
\label{ex:radCM}
Consider the monomial ideal
\[
I = \< \del_1\del_4, \del_2\del_4, \del_2\del_3, 
\del_1\del_3\del_5, \del_5^2 \> \subseteq \CC[\del_1,\dots,\del_5].
\]
The irreducible components of $\VV(\tilde{I})$ are
\[
\CC^{\{1,2\}}, \;\; \CC^{\{1,3\}}, \;\; \CC^{\{3,4\}}, \;\; 
(0,0,0,0,1) + \CC^{\{1,2\}}, \;\;
(0,0,0,0,1) + \CC^{\{3,4\}}.
\]
For $b=(0,0,0,0,1)$,
the simplicial complex $\Delta_b(I)$ consists of two line segments, 
$\{ 1,2\}$ and $\{3,4 \}$.
The link of the empty set is all of $\Delta_b(I)$,
a one-dimensional simplicial complex
with nonvanishing zeroth reduced homology;
this implies that $\Delta_b(I)$ is not Cohen--Macaulay
by Theorem~\ref{thm:reisner}.
Thus, while $I$ is unmixed and $\sqrt{I}$ is Cohen--Macaulay,
Theorem~\ref{thm:main} implies that $I$ is not Cohen-Macaulay.
\end{example}

By introducing new variables, 
one can pass from the monomial ideal $I \subseteq \CC[\del]$ 
to its \emph{polarization}, a squarefree monomial ideal $I_\Delta$ 
in a larger polynomial ring $S$ such that 
$S/(I_\Delta+\< y \>) \cong \CC[\del]/I$ 
for some regular sequence $y$ in $S/I_\Delta$. 
In particular, $I$ is Cohen--Macaulay if and only if 
the single simplicial complex $\Delta$ is Cohen--Macaulay. 
However, the number of variables in $S$ can make the application of 
Theorem~\ref{thm:reisner} to $\Delta$ much more 
computationally expensive than checking the Cohen--Macaulayness 
of the finitely many $\Delta_b(I)$ of Theorem~\ref{thm:main}. 
This is the case in the following family of monomial ideals 
from \cite{cm examples}.

\begin{example}
\label{ex:polarization}
Given $k \in \NN$, let $m_i = \prod_{j \neq i} \del_j^k$, and
consider the ideal 
\[
I_k = \< m_1,\dots, m_n\> \subseteq \CC[\del_1,\dots,\del_n].
\]
Jarrah shows that $\CC[\del]/I_k$ is Cohen--Macaulay
by explicitly constructing the minimal free resolution. 
Let us verify this fact combinatorially. 

By Theorem~\ref{thm:main}, we only need
to check that $\CC[\del]/I_1$ is Cohen--Macaulay,
as the exponent complexes of $I_k$ and $I_1$ are the same.
But $I_1$ is squarefree, so it is enough to show
that $\Delta_0(I_1)$ is a Cohen--Macaulay simplicial 
complex (of dimension $n-3$). 
Now
\[
\Delta_0(I_1)=
\{ 
\sigma \subseteq \{1,\dots,n\} \ \big| \ |\sigma| \leq n-2
\}.
\]
The link of $\sigma \in \Delta_0(I_1)$ consists of  
subsets of $\{1,\dots,n\}\minus \sigma$
of cardinality at most $n-2-|\sigma|$. 
This simplicial complex is homotopy equivalent to a
wedge of $(n-3-|\sigma|)$-spheres, whose
only nonvanishing reduced homology occurs in the top degree.
Thus $\Delta_0(I_1)$ is Cohen--Macaulay.

On the other hand, the polarization of $I_k$ is much less transparent.
For instance, if $k=4$ and $n=4$, the polarization of 
$I_4=\<\del_1^4\del_2^4\del_3^4,\del_1^4\del_2^4\del_4^4,
\del_1^4\del_3^4\del_4^4,\del_2^4\del_3^4\del_4^4 \>$ is  
\begin{align*}
I_{\Delta} = \
& \< 
s_1s_2s_3s_4t_1t_2t_3t_4u_1u_2u_3u_4, 
s_1s_2s_3s_4t_1t_2t_3t_4v_1v_2v_3v_4, \\
& \ \ s_1s_2s_3s_4u_1u_2u_3u_4v_1v_2v_3v_4, 
t_1t_2t_3t_4u_1u_2u_3u_4v_1v_2v_3v_4 
\> \\
& \subseteq \CC[s_1,s_2,s_3,s_4,t_1,t_2,t_3,t_4,u_1,u_2,u_3,u_4,v_1,v_2,v_3,v_4], 
\end{align*}
where, for example, 
$\del_1^4$ has been replaced by $s_1s_2s_3s_4$. 
The $f$-vector $(f_{-1},f_0,f_1,\dots)$ of $\Delta$ is 
\[
f(\Delta) = 
(1,16, 120, 560, 1820, 4368, 8008, 11440, 12870, 11440, 8008, 4368, 1816, 544, 96), 
\]
where $f_i$ is the number of faces of $\Delta$ of dimension $i$. 
Applying Reisner's criterion
to verify the Cohen--Macaulayness of
$\Delta$ is computationally expensive, 
especially in comparison with verification via 
Theorem~\ref{thm:main}, which involves only $1$-dimensional
simplicial complexes on four vertices.
\end{example}

\begin{remark}
\label{remark:takayama complexes}
There are several known families of simplicial complexes 
that capture homological properties of $I$,
including those in \cite[Theorems~5.3.8 and~5.5.1]{bruns-herzog}, 
\cite{mustata}, \cite[Theorems~1.34 and~5.11]{MS05}, and \cite{takayama05}. 
See \cite{miller-letter} for details on how these relate to each other
and to the exponent complexes $\Delta_b(I)$.
\end{remark}

\section{A combinatorial  formula for rank }
\label{sec:primary complex}

For the remainder of this article, we work over any field $\kk$ of characteristic zero. 
The $\kk$-vector space dimension of
$\kk[\theta]/(\tilde{I}+ \< E-\beta \>)$
measures the deviation (with respect to $\beta \in \kk^d$)
of the ideal $I$ from being Cohen--Macaulay.
The goal of this section is to explicitly compute
this dimension, which equals $\rank(I+\<E-\beta\>)$
when $\kk =\CC$, 
by Lemma~\ref{lemma:all-holonomic}.

We first observe that the dimension we wish to compute is equal to the sum over
the exponents $b\in \VV(\tilde{I} + \< E - \beta \>)$ of
the $\kk$-vector space dimensions of
$\kk[\theta]/(I_{\Delta_b(I)}+\< E \>)$. 
This reduces the computation to the case that
$\tilde{I}$ is squarefree and $\beta = 0$.

\begin{proposition}
\label{thm:dim count general}
Let $I$ be a monomial ideal in $\kk[\del]$, not necessarily squarefree.
Then
\[
\dim_\kk \left( \frac{\kk[\theta]}{\tilde{I} + \< E - \beta \>} \right)
    = \sum_{b \in \VV(\tilde{I} +  \< E - \beta \>)} \dim_\kk
    \left( \frac{\kk[\theta]}{ I_{\Delta_b(I) } + \< E \>} \right).
\]
\end{proposition}
\begin{proof}
Recall that for $b \in \VV( \tilde{I} + \< E - \beta \> )$,
$I_{\Delta_b(I)} = J_b \subseteq \kk[\theta]$
denotes the Stanley--Reisner ideal of $\Delta_b(I)$.
Since $\tilde{I} + \< E - \beta \> =
\bigcap_{b \in \VV(\tilde{I} +  \< E - \beta \>)} J_b(\theta - b) + \< E - \beta \>$,
the $\kk$-vector space dimension of
$\kk[\theta]/( \tilde{I} + \< E - \beta \> )$
is equal to the sum over all exponents of $I$ with respect to $\beta$
of the dimensions of
$\kk[\theta]/(J_b(\theta - b) + \< E - \beta \>)$.
Each of these clearly coincides with the dimension of
$\kk[\theta]/( J_b + \< E\>)$. 
\end{proof}

\subsection{The squarefree case}
\label{subsec:squarefree rank}

Henceforth,
$I = I_{\Delta}$ is the squarefree monomial ideal in $\kk[\theta]$ 
corresponding to a simplicial complex $\Delta$ on $\{1,\dots,n \}$.

\begin{notation}
Let $F^0$ denote the set of facets of $\Delta$, and for $p>0$, set 
\begin{align}
F^p = \{ s\subseteq F^0 \ \big| \ |s| = p+1 \}.
\end{align}
Each element $s \in F^p$ determines a face
\[
\sigma(s) = \bigcap_{\tau \in s} \tau \ \in \Delta.
\]
Note that $\sigma(F^0) = \bigcap_{s\in F^0} \sigma(s)$ is 
the face of $\Delta$ corresponding to the unique element of $F^{|F^0|-1}$.
For $S\subseteq F^p$, we denote by 
\begin{align}
\label{kappa}
\kappa(S) = 
\max \left(0,\ d - \bigg| \bigcup_{s\in S} \sigma(s) \bigg| \right),
\end{align}
%
(see also Notation~\ref{not:asigma}).
Finally, we give notation to describe the $(p+1)$-intersections 
of facets of $\Delta$ that have dimension less than $d-p$: 
\begin{align}
G^p = \{ s\in F^p \mid \kappa(s) \geq p+1 \}.
\end{align} 
\end{notation}

\begin{definition}
\label{def: circuit}
The collections $F^\bullet$ naturally 
correspond to the nonempty faces of a $(|F^0|-1)$-simplex $\Omega$; 
we may view a subset $S \subseteq G^p$ 
as a collection of $p$-faces of $\Omega$.
Given $t\in G^p\setminus S$, 
we abuse notation and write $S$ and $S\cup\{t\}$ 
for the subcomplexes of $\Omega$ 
whose maximal faces are given by these sets.
Suppose that there is a minimal generator of 
$H^p( S\cup\{t\}, S; \kk)$ of the form 
$\sum_{s\in S} v_s\cdot s + v_t\cdot t$, 
where all coefficients are nonzero.
In this case, we say that $S\cup\{t\}$ is a \emph{circuit for $t$}.
\end{definition}

\begin{notation}
\label{not: psi} 
Let $G^p = \{s_1, s_2, \dots, s_{|G^p|} \}$ be a fixed order on $G^p$. 
For $1\leq j< |G^p|$, set 
\begin{align*}
G^p(j) & = \{ s_1,\dots, s_j \},\\
L^p(j) & = \{ \tau\subseteq \{1,\dots,j+1\} \mid
	 \{s_i\}_{i\in \tau} \text{ is a circuit for $s_{j+1}$}, \  j+1 \in \tau \}, \\
\text{ and } \quad
L^p(j,k) & = \{ \Lambda \subseteq L^p(j) 
	\mid |\Lambda| = k \} \text{ for $1\leq k \leq |L^p(j)|$. } 
\end{align*}
For $\Lambda\in L^p(j,k)$, 
let $s_\Lambda = \bigcup_{\tau\in \Lambda} \bigcup_{i\in \tau} s_i$, 
and define 
\begin{align*}
\psi^p(j) = &
\sum_{k=1}^{|L^p(j)|}
	\sum_{\Lambda \in L^p(j,k)}  
	(-1)^{|\Lambda|+1}
	\binom{ \kappa( s_\Lambda) }{p+1} 
\end{align*}
Adding over all $j$, we obtain 
\[
\psi^p = \sum_{j = 1}^{|G^p| - 1} \psi^p(j).
\]
This notation is designed for the proof of 
Lemma~\ref{lemma:second dim}, where its motivation will become clear. 
\end{notation}

The derivation of the following formula will occupy the remainder of this section. 

\begin{theorem}
\label{thm:dim count squarefree}
If $I = I_\Delta$ is a squarefree monomial ideal in $\kk[\theta]$, then
\smallskip
\begin{align*}
\dim_\kk \left( \frac{\kk[\theta]}{ I + \< E\>} \right)
= 
\deg(I) 
	+ \binom{\kappa( \sigma(F^0) )-1}{|F^0|-1} 
	- \sum_{p=0}^{|F^0|-2} \sum_{s\in G^p} \binom{\kappa(s)-1}{p+1}
	+ \sum_{p=0}^{|F^0| - 2} \psi^p. 
\end{align*}
\end{theorem}
\medskip

\begin{remark}
\label{rem:deg in formula}
The set of top-dimensional faces of $\Delta$ is 
$F^0\setminus G^0 = \{ s\in F^0 \mid |\sigma(s)| = d\}$ 
because $\dim(\Delta) = d-1$, 
and the degree of $I$ is
\[
\deg(I) = 
\sum_{s\in F^0\setminus G^0} \binom{\kappa(s)}{0} = |F^0\setminus G^0|.
\]
\end{remark}

\smallskip
\begin{example}[Example~\ref{ex:radCM}, continued]
\label{ex:radCM rank}
With $b=(0,0,0,0,1)$, 
consider the squarefree ideal 
$I_{\Delta_b(I)} \subseteq \kk[\theta_1,\dots,\theta_5]$.
Recall that the simplicial complex $\Delta_b(I)$ 
consists of two line segments,
$\{ 1,2 \}$ and $\{ 3,4 \}$,
so $F^0 = \{ s = \{ 1,2 \}, s' = \{ 3,4 \} \}$,
$F^1 = \{ \{ s,s'\} \}$,
and $F^p = 0$ for $p\neq 0,1$.
Also, 
$G^0 = \nothing$. 
By Theorem~\ref{thm:dim count squarefree},
\begin{align*}
\dim_\kk \left( \frac{\kk[\theta]}{ I_{\Delta_b(I)} + \< E\>} \right) =
& \ 
2 + \binom{1}{1} - 0 + 0
= 3.
\end{align*}
\end{example}

\medskip

To prove Theorem~\ref{thm:dim count squarefree}, 
we begin by constructing an acyclic cochain complex, 
called the \emph{primary resolution} of $I = I_\Delta$ because its only 
homology module is isomorphic to $\kk[\theta]/I$. 
We will then consider the resulting Koszul spectral sequence with respect to $E$.

\begin{notation}
\label{not:asigma}
The submatrix of $A$ with columns corresponding to a face $\sigma \in \Delta$
is denoted $A_\sigma = (a_i)_{i\in\sigma}$, and
$\kk[\theta_\sigma] = \kk[\theta_i \mid i \in \sigma]$
is a polynomial ring in the corresponding variables.
We view $\kk[\theta_\sigma]$ as a $\kk[\theta]$-module
via the natural surjection
$\kk[\theta] \twoheadrightarrow
\kk[\theta]/\< \theta_i \mid i \notin \sigma \> = \kk[\theta_\sigma]$.
%
%
%
%
For $s\in F^p$, 
$\dim \kk[\theta_{\sigma(s)}] = | \sigma(s) | = \dim_\kk ( \kk A_{\sigma(s)} )$ 
by our assumptions on $A$ in Convention~\ref{conv:A}. 
Thus $\kappa(s)$ from Notation~\ref{kappa} is equal to 
$\codim_{\kk^d}(\kk A_{\sigma(s)})$. 
\end{notation}

By choice of a suitable incidence function on the lattice $F^\bullet$,
there is an exact sequence of $\kk[\theta]$-modules: 
\begin{eqnarray}
\label{eqn:primary complex}
0 \rightarrow \kk[\theta]/I \rightarrow R^\bullet,
\end{eqnarray}
where
\begin{eqnarray}
R^p = \bigoplus_{s \in F^p} \kk[\theta_{\sigma(s)}].
\end{eqnarray}
We call $R^\bullet$ the \emph{primary resolution} of $I$. 
We may view $R^\bullet$ as a 
cellular resolution supported on an $(|F^0| - 1)$-simplex 
(see \cite[Definition~6.7]{rank jumps}). 

\begin{example}[Example~\ref{ex:radCM rank}, continued]
In this case, the primary resolution of $I$ will be
\[
\kk[\theta_1,\theta_2] \oplus \kk[\theta_3,\theta_4]
    \stackrel{\delta}{\longrightarrow} \kk \longrightarrow 0
\]
with differential $\delta(f,g) = \ol{f}-\ol{g}$,
where $\ol{ \phantom{} \cdot \phantom{} }$
denotes image modulo $ \<\theta\>$.
\end{example}

\begin{remark}
The primary resolution $R^\bullet$ of (\ref{eqn:primary complex})
is an example of an \emph{irreducible resolution} of $\kk[\theta]/I$
in \cite[Definition~2.1]{cm quotients}.
\cite[Theorem~4.2]{cm quotients}
shows that $I$ is Cohen--Macaulay
precisely when $R^\bullet$ contains an
irreducible resolution $S^\bullet$ of $\kk[\theta]/I$
such that $S^p$ is of pure Krull dimension
$\dim(I) - p$ for all $p\geq 0$.
If this is the case, then by Corollary~\ref{cor:rank CM}, 
the $\kk$-vector space dimension of $\kk[\theta]/( I + \< E\>)$
is equal to the degree of $I$.
However, our current computation of $\dim_\kk(\kk[\theta]/( I + \< E\>))$
is independent of this remark.
\end{remark}

\begin{remark}
While the primary resolution $R^\bullet$ of $I$ may appear to be similar
to the complex in \cite[Theorem~5.7.3]{bruns-herzog},
they are generally quite different.
For each $\sigma$, this other complex places
$\kk[\theta_\sigma]$ in the
$|\sigma|$-th homological degree.
However, $R^\bullet$ permits summands
of rings of varying Krull dimension within a
single cohomological degree
and $\kk[\theta_\sigma]$ never appears in
the primary resolution if $\sigma$ is not contained
in the intersection of a collection of facets of $\Delta$.
In fact, the primary resolution
will coincide with (a shifted copy of)
the complex in \cite[Theorem~5.7.3]{bruns-herzog}
only when $\Delta$ is the boundary of a simplex.
\end{remark}

Consider the double complex $E_0^{\bullet,\bullet}$ given by
$E_0^{p,-q} = K_q( R^p; E)$,
where $K_\bullet(-;E)$ forms a Koszul complex 
with respect to the sequence $E$.
Taking homology first with respect to the horizontal differential, we see that
\begin{eqnarray*}
{}_h E_\infty^{p,-q} = {}_h E_1^{p,-q} =
\begin{cases}
    K_q( \kk[\theta]/I; E) &  \text{if $p=0\text{ and }0\leq q \leq d,$}\\
    0 & \text{otherwise.}
\end{cases}
\end{eqnarray*}
Let ${}_{v}E_\bullet^{\bullet,\bullet}$ denote
the spectral sequence obtained from $E_0^{\bullet,\bullet}$
by first taking homology with respect to
the vertical differential.
Since $H_0( \kk[\theta]/I; E ) = \kk[\theta]/(I + \< E \>)$ and
${}_{v}E_\bullet^{\bullet,\bullet}$
converges to the same abutment as ${}_{h}E_\bullet^{\bullet,\bullet}$,
\begin{eqnarray}
\label{eqn:to IIE}
\dim_\kk \left( \frac{ \kk[\theta] }{ I + \< E \> } \right)
= \sum_{p-q=0} \dim_\kk ({}_{v}E_\infty^{p,-q}).
\end{eqnarray}
Notice that 
\begin{eqnarray*}
{}_{v}E_1^{p,-q} =
\begin{cases}
    H_q( R^p; E ) &  \text{if $0\leq q \leq d,$}\\
    0 & \text{otherwise,}
\end{cases}
\end{eqnarray*}
and for $0\leq q\leq d$,
\begin{eqnarray*}
H_q( R^p; E ) = \bigoplus_{s \in F^p} H_q( \kk[\theta_{\sigma(s)}]; E ).
\end{eqnarray*}

Instead of ${}_{v}E_\bullet^{\bullet,\bullet}$,
we will study a sequence ${'E}_\bullet^{\bullet,\bullet}$ with
the same abutment and differentials which behave well
with respect to vanishing of Koszul homology.
To introduce this sequence, we first need some notation.

\begin{notation}
For $\sigma \in \Delta$,
let $L(\sigma)$ be the lexicographically first subset of
$\{1,2,\dots,d\}$ of cardinality equal to $\dim_\kk (\kk A_\sigma)$
such that $\{ E_i \}_{i\in L(\sigma)}$
is a system of parameters for $\kk[\theta_\sigma]$ (as a $\kk[\theta]$-module).
For a $\ZZ$-graded $\kk[\theta_\sigma]$-module $M$,
let $K^\sigma_\bullet( M; E )$ denote the Koszul complex on $M$
given by the operators $\{E_i \}_{i\in L(\sigma)}$ and
$H^\sigma_q( M; E ) = H_i( K^\sigma_\bullet( M; E ) )$.
\end{notation}

\begin{notation}
For $\sigma \in \Delta$, let
\[
\mbox{
$\ker_\ZZ A_\sigma = 
   \{ v \in \ZZ^d \mid \sum_{i=1}^d v_i a_{ij} = 0 \ \forall j \in \sigma \},
$}
\]
using the standard basis of $\ZZ A = \ZZ^d$,
and let $\bigwedge^\bullet (\ker_\ZZ A_\sigma)$
denote a complex with trivial differentials.
These complexes will be useful in our
computation of $\dim_\kk (\kk[\theta]/( I + \< E \> ))$.
\end{notation}

\begin{lemma}
\label{lemma:EKdecomp}
Let $\sigma \in \Delta$ and $M$ be an $A$-graded $\kk[\theta_\sigma]$-module.
There is a quasi-isomorphism of complexes
\begin{align}
\label{eqn:EK qis}
\mbox{
	$K_\bullet( M; E) \simeq_{\text{qis}}
	K^\sigma_\bullet( M; E )
	\otimes_\ZZ  \bigwedge^\bullet \left( \ker_\ZZ A_\sigma \right).$
}
\end{align}
In particular, there is a decomposition
$
H_\bullet(\kk[\theta_\sigma]; E ) \cong
H^\sigma_0(\kk[\theta_\sigma]; E )
\otimes_\ZZ \bigwedge^\bullet \left(\ker_\ZZ A_\sigma \right).
$

\end{lemma}
\begin{proof}
The quasi-isomorphism (\ref{eqn:EK qis}) is a commutative version
of Proposition~3.3 in \cite{rank jumps}.
Since $\dim_\kk( \kk A_\sigma ) = | \sigma | = \dim \kk[\theta_\sigma]$,
the choice of $L(\sigma)$ implies that
$\{E_i\}_{i\in L(\sigma)}$ is a maximal $\kk[\theta_\sigma]$-regular sequence.
Thus the second statement follows from \cite[Corollary~1.6.14]{bruns-herzog}.
\end{proof}

\begin{lemma}
\label{lemma:EKimage}
Let $\sigma \subseteq \tau$ be faces of $\Delta$, and let
$\pi: \kk[\theta_\tau] \twoheadrightarrow \kk[\theta_\sigma]$
be the natural surjection.
There is a commutative diagram
\begin{eqnarray}
\xymatrix{
K_\bullet(M; E ) \ar[r]^{K_\bullet(\pi; E )} \ar[d] &
K_\bullet(N; E ) \ar[d] \\
K^{\tau}_\bullet(M; E ) \otimes_\ZZ
     \bigwedge^\bullet \left( \ker_\ZZ A_\tau \right) \ar[r] &
K^{\sigma}_\bullet(N; E )
    \otimes_\ZZ \bigwedge^\bullet \left( \ker_\ZZ A_\sigma \right)
}
\end{eqnarray}
with vertical maps given by (\ref{eqn:EK qis}). 
Further, the image of $H_\bullet(\pi; E)$ is isomorphic to
\[
\mbox{
$H^\sigma_0(\kk[\theta_\sigma]; E )
    \otimes_\ZZ \bigwedge^\bullet \left( \ker_\ZZ A_\tau \right)$ }
\]
as a submodule of
$H^{\sigma}_0(N; E ) \otimes_\ZZ
    \bigwedge^\bullet \left( \ker_\ZZ A_\sigma \right)$.
\end{lemma}
\begin{proof}
This is a modified version of
Lemma~3.7 and Proposition~3.8 in \cite{rank jumps}.
\end{proof}

\begin{notation}
\label{not:delta}
Let $({'E}_\bullet^{\bullet,\bullet},\delta_\bullet^{\bullet,\bullet})$ 
be the spectral sequence determined by 
letting $\delta_0^{\bullet,\bullet}$ be the vertical differential of 
the double complex
${'E}_0^{\bullet,\bullet}$ with
\begin{align*}
\mbox{
${'E}_0^{p,-q} =
    \bigoplus_{ s \in F^p } \bigoplus_{ i + j = q }
    K^\sigma_i( \kk[\theta_{\sigma(s)}]; E )
    \otimes_\ZZ \bigwedge^j \left( \ker_\ZZ A_{\sigma(s)} \right). 
$} \end{align*}
Note that by Lemma~\ref{lemma:EKdecomp},
\begin{align}
\label{eqn:ss}
\mbox{
${'E}_1^{p,-q} =
    \bigoplus_{ s \in F^p } H^\sigma_0( \kk[\theta_{\sigma(s)}]; E )
    \otimes_\ZZ \bigwedge^q \left( \ker_\ZZ A_{\sigma(s)} \right).
$}
\end{align}
\end{notation}

\smallskip
By Lemmas~\ref{lemma:EKdecomp}~and~\ref{lemma:EKimage},
the horizontal differentials of $E_0^{\bullet,\bullet}$
are compatible with the quasi-isomorphism
\begin{eqnarray*}
\mbox{
$E_0^{p,\bullet} \simeq_{qis}
    \bigoplus_{s\in F^p} K^\sigma_\bullet( \kk[\theta_{\sigma(s)}]; E )
    \otimes_\ZZ \bigwedge^\bullet \left( \ker_\ZZ A_{\sigma(s)} \right).
$}\end{eqnarray*}
Thus we may replace ${}_{v}E_\infty^{p,-q}$ in (\ref{eqn:to IIE})
with  ${'E}_\infty^{p,-q}$, obtaining 
\begin{eqnarray}
\label{eqn:use 'E}
\dim_\kk \left( \frac{\kk[\theta]}{ I + \< E \> } \right)
= \sum_{p-q=0} \dim_\kk ( {'E}_\infty^{p,-q}).
\end{eqnarray}
This replacement is beneficial because
${'E}_\bullet^{\bullet,\bullet}$ degenerates quickly.

\begin{lemma}
\label{lemma:ss terminates}
The spectral sequence ${'E}_\bullet^{\bullet,\bullet}$ of
(\ref{eqn:ss}) degenerates at the second page:
$'E_2^{\bullet,\bullet}= {'E}_\infty^{\bullet,\bullet}.$
\end{lemma}
\begin{proof}
The proof follows the same argument as \cite[Lemma~5.26]{rank jumps}.
\end{proof}

Lemma~\ref{lemma:ss terminates} and (\ref{eqn:use 'E}) imply that
\begin{equation}
\label{eqn:use 'E1}
\dim_\kk \left( \frac{ \kk[\theta]}{I + \< E \>} \right)
= \sum_{ p - q \geq 0 } (-1)^{ p - q } \dim_\kk ({'E}_1^{p,-q})
    - \sum_{ p - q = - 1 } \dim_\kk (\image \delta_1^{p,-q}), 
\end{equation}
and recall from Notation~\ref{not:delta} that $\delta_\bullet^{\bullet,\bullet}$ 
denotes the differential of
${'E}_\bullet^{\bullet,\bullet}$. 
We compute the dimensions on the right hand side of (\ref{eqn:use 'E1})
in the following lemmas.

\begin{lemma}
\label{lemma:first dim}
For $0 \leq q \leq d$ and $p \geq 0$,
\[
\dim_\kk ({'E}_1^{p,-q}) = \sum_{s \in F^p} \binom{ \kappa(s) }{q}.
\]
\end{lemma}
\begin{proof}
By Lemma~\ref{lemma:EKdecomp},
\[
\mbox{
${'E}_1^{p,-q} =
    \bigoplus_{s \in F^p} H^\sigma_0( \kk[\theta_{\sigma(s)}]; E )
    \otimes_\ZZ  \bigwedge^q \left( \ker_\ZZ A_{\sigma(s)} \right).$
 }
\]
Since $\{ E_i \}_{i\in L_{\sigma(s)}}$
is a system of parameters for $\kk[\theta_{\sigma(s)}]$ for all $s \in F^p$,
the $\kk$-vector space dimension of the module 
$H^\sigma_0( \kk[\theta_{\sigma(s)}]; E )$ is $1$.
Hence
\[
\dim_\kk  ({'E}_1^{p,-q}) =
    \sum_{s \in F^p} \binom{ \codim_{\kk^d} ( \kk A_{\sigma(s)} ) }{q},
\]
and by the choice of $A$ in Convention~\ref{conv:A},
$\kappa(s) = \codim_{\kk^d} ( \kk A_{\sigma(s)} )$.
\end{proof}

\begin{lemma}
\label{lemma:second dim}
For $p \leq |F^0|-2$,
\begin{eqnarray}
\dim_\kk (\image \delta_1^{p,-p-1})
=  - \psi^p + \sum_{s\in G^p} \binom{ \kappa(s) }{p+1}. 
\end{eqnarray}
\end{lemma}
\begin{proof}
We first note that if $p > |F^0|-2$, 
then $R^{p+1} = 0$ and $\image \delta_1^{p,-p-1} = 0$. 

Now with $0 \leq p \leq |F^0| - 2$, 
fix an order for the elements of $G^p$, say 
\[
G^p = \{s_1, s_2, \dots, s_{|G^p|} \}, 
\]
as in Notation~\ref{not: psi}. 
For a subset $S\subseteq G^p$, 
let $\delta_{1,S}^{p,-q}$ denote the restriction of 
$\delta_{1,S}^{p,-q}$ to the summands of (\ref{eqn:ss}) that lie in $S$.
We use this notation to see that 
\begin{align}
\nonumber \dim_\kk & (\image \delta_1^{p,-p-1}) \\
& = 
	\sum_{i=1}^{|G^p|} \dim_\kk (\image \delta_{1,s_i}^{p,-p-1}) 
		- \sum_{j = 1}^{|G^p| - 1} \dim_\kk \left[
			( \image \delta_{1,\{ s_1, \dots, s_j \} }^{p,-p-1} )
			\cap ( \image \delta_{1, s_{j+1} }^{p,-p-1} )
		\right]. \label{eqn:linear algebra} 
\end{align}		

By Lemma~\ref{lemma:EKimage} and (\ref{eqn:ss}), for $s\in G^p$, 
\[
\dim_\kk (\image \delta_{1,s}^{p,-p-1}) 
= \binom{\kappa(s)}{p+1} \cdot \dim_\kk (\image \delta_{1,s}^{p,0}) 
= \binom{\kappa(s)}{p+1}.
\]
To compute the value of the terms in the second summand of 
(\ref{eqn:linear algebra}), we start by identifying 
a spanning set of the vector space 
\begin{align}
\label{eqn:intersection}
( \image \delta_{1,\{ s_1, \dots, s_j \} }^{p,-p-1} )
			\cap ( \image \delta_{1, s_{j+1} }^{p,-p-1} ).
\end{align} 
Consider  
$( \image \delta_{1,\{ s_1, \dots, s_j \} }^{p,0} )
			\cap ( \image \delta_{1, s_{j+1} }^{p,0} )$, 
which has dimension either 0 or 1. 
If this is 1, then 
by Lemma~\ref{lemma:EKimage} and (\ref{eqn:ss}), 
nonzero intersections of the form (\ref{eqn:intersection}) 
arise from circuits $s_I$ of $s_{j+1}$.
Further, the contribution of such a circuit $s_I$ to the dimension 
of~(\ref{eqn:intersection}) is the $\ZZ$-rank of 
\begin{align*}
\bigcap_{i\in I} \left[
	 \bigwedge^{p+1} \left(\ker_\ZZ A_{\sigma(s_i)} \right) 
\right] 
= \bigwedge^{p+1} \left( \ker_\ZZ 
A_{\mbox{\tiny $\bigcup_{i\in I} \sigma(s_i)$
}}
\right). 
\end{align*}
This is precisely $\binom{\kappa(s_I)}{p+1}$, 
by definition of $\kappa(-)$ in (\ref{kappa}).
Now applying the inclusion-exclusion principle to account for overlaps, 
\begin{align*}
\dim_\kk  (\image \delta_1^{p,-p-1}) 
& =	\sum_{i=1}^{|G^p|} \binom{ \kappa(s_i) }{p+1} 
		- \sum_{j = 1}^{|G^p| - 1} \psi^p(j), 
\end{align*}
and we obtain the desired formula by Notation~\ref{not: psi}.
\end{proof}

\begin{proof}[Proof of Theorem~\ref{thm:dim count squarefree}]
%
%
Combining (\ref{eqn:use 'E1}) and 
Lemmas~\ref{lemma:first dim} and \ref{lemma:second dim}, 
\begin{align*}
\dim_\kk \left( \frac{\kk[\theta]}{ I + \< E\>} \right) 
= &
\sum_{p-q \geq 0} \sum_{s \in F^p} (-1)^{p-q}
 \binom{ \kappa(s) }{q} 
- \sum_{p=0}^{|F^0|-2} \sum_{s\in G^p} \binom{ \kappa(s) }{p+1} 
+ \sum_{p=0}^{|F^0| - 2} \psi^p \\
= & 
\sum_{s\in F^0\setminus G^0} \binom{\kappa(s)}{0} 
	+ \sum_{s\in G^0} \binom{ \kappa(s)-1}{0} \\
	& \quad + \sum_{p=1}^{|F^0|-1} \sum_{s \in F^p} \binom{\kappa(s)-1 }{p} 
	- \sum_{p=0}^{|F^0| - 2} \sum_{s\in G^p} \binom{ \kappa(s) }{p+1}  
	+ \sum_{p=0}^{|F^0| - 2} \psi^p \\
= & \  
|F^0\setminus G^0|	 
	+ \sum_{p=1}^{|F^0|-2} \sum_{s \in F^p\setminus G^p} \binom{ \kappa(s)-1 }{p}  
	+ \binom{ \kappa( \sigma(F^0) )-1}{|F^0|-1} \\
	& \quad - \sum_{p=0}^{|F^0|-2} \sum_{s\in G^p} \binom{ \kappa(s)-1}{p+1}
	+ \sum_{p=0}^{|F^0| - 2} \psi^p, 
\end{align*}
Recall from Remark~\ref{rem:deg in formula} that 
$|F^0\setminus G^0| = \deg(I)$. 
Further, if $s\in F^p\setminus G^p$, 
then $\kappa(s)-1<p$, and we obtain the desired formula. 
\end{proof}

\raggedbottom
\def\cprime{$'$} \def\cprime{$'$}
\providecommand{\MR}{\relax\ifhmode\unskip\space\fi MR }
\providecommand{\MRhref}[2]{%
  \href{http://www.ams.org/mathscinet-getitem?mr=#1}{#2}
}
\providecommand{\href}[2]{#2}

\end{document}